\documentclass[a4paper,abstracton,11pt]{article}
\usepackage{amsfonts}
\usepackage{amssymb}
\usepackage{amsmath}
\usepackage{mathtools}
\usepackage{amsthm}
\usepackage{xspace}
\usepackage{apacite}

\usepackage{accents}
\usepackage[normalem]{ulem}

\usepackage{booktabs}

\usepackage{natbib}

\usepackage{graphicx}
\usepackage{xcolor}

\usepackage{subfigure}

\usepackage{a4wide}
\usepackage{hyperref}

\usepackage{setspace}
\usepackage{authblk}

\allowdisplaybreaks

\newtheorem{theorem}{Theorem}

\newtheorem{cor}[theorem]{Corollary}

\newtheorem{example}[theorem]{Example}

\newcommand{\X}{{\mathcal{X}}}
\newcommand{\I}{{\mathcal{I}}}
\newcommand{\XI}{{\X(I)}}
\newcommand{\x}{{\pmb{x}}}

\begin{document}

\title{A Framework for Data-driven Explainability in Mathematical Optimization}

\author[1]{Kevin-Martin Aigner}
\author[2]{Marc Goerigk}
\author[3]{Michael Hartisch}
\author[1]{Frauke Liers}
\author[1]{Arthur Miehlich}

\affil[1]{Department of Data Science, Friedrich-Alexander-Universität Erlangen-Nürnberg, Germany}

\affil[2]{Business Decisions and Data Science, University of Passau, Germany}
\affil[3]{Network and Data Science Management, University of Siegen, Germany}

\date{}

\maketitle

\begin{abstract}
Advancements in mathematical programming
have made it possible to efficiently tackle large-scale real-world
problems that were deemed intractable just a few decades ago. However,
provably optimal solutions may not be accepted due
to the perception of optimization software as a black box. Although
well understood by scientists, this lacks easy accessibility for practitioners.
Hence, we advocate for introducing the explainability of a solution as
another evaluation criterion, next to its objective value, which enables us
to find trade-off solutions between these two criteria.
Explainability is attained by comparing against (not necessarily optimal)
solutions that were implemented in similar situations in the past.
Thus, solutions are preferred that exhibit
similar features. Although we prove that
already in simple cases the explainable model is NP-hard, we
characterize relevant polynomially solvable cases such as the
explainable shortest path problem. Our numerical experiments on
both artificial as well as real-world road networks show the
resulting Pareto front. It turns out that the cost of enforcing explainability can be very small. 
\end{abstract}

\section{Introduction}

\subsection{Motivation}

The area of mathematical optimization plays a vital role in decision making across many fields, from  enhanced process designs, making efficient use of scarce resources, to aids in financial planning.
For example, transportation and logistics benefit from optimized routes and deliveries, and modern challenges in energy management and in environmental planning can be addressed.

The last decades have witnessed remarkable improvement with respect to efficient algorithms and implementations of optimization methods, making it possible to solve huge optimization tasks to provably global optimality within short time.
On the one hand, this is true for polynomial-time solvable optimization problems where modern algorithm engineering approaches have significantly reduced solution times in practice.
On the other hand, significant progress has also been made in enhancing branch-and-bound and branch-and-cut algorithms for the global solution of challenging, NP-hard, mathematical optimization problems.
These improved algorithms employ advanced insights into the structure of the underlying problems, sophisticated strategies to efficiently explore the solution space and effectively prune infeasible or non-optimal branches.
As a result, nowadays even large instances can be solved to global optimality within reasonable solution time.

Until now, the research focus has primarily been on obtaining solutions that can be shown to be provably optimal.
In contrast, in particular when human beings are involved in the resulting decision-making process, it is also of crucial importance to understand the solution itself.
In the realm of problem-solving, the explainability of a solution can actually be similarly important as, or even more important than, its optimality.
While achieving optimal results is undoubtedly important, the ability to understand and interpret a solution is equally vital.
Indeed, gaining insights and comprehending the underlying decisions can lead to improved transparency and thus to improved acceptance and trust in an obtained solution.

Despite this fact, the explainability of mathematical programming models and algorithms has hardly received any attention in the past.
One reason for this might be that experts in the field view
mathematical optimization models
as inherently interpretable, as they appear to be already transparent and easily comprehensible.
Furthermore, many algorithms for solving or approximating solutions of
such models are well established, being the result of several decades of successful algorithmic
developments. Hence, researchers in operations
research and optimization often have high confidence in the procedures. 

Obtained solutions themselves are typically not
understandable, and one can easily wonder why some particular solution has been
obtained. This is even more true for non-experts that have to realize
the optimized solution or are affected by the solution: planners have
to justify the allocation of capacities in a network, factory workers
want to understand why they have to take a disliked shift, airports
need to explain to airlines why some incoming or
outgoing aircraft had to be delayed, and solar panel owners want (and might even have the right to) an explanation of why the power feed of their solar panels was throttled.
Providing the used optimization model without extensive explanations
or pointing at correctness proofs of algorithms is certainly not
enough to provide satisfactory explanations, to build trust or to promote accountability. 

Our work contributes a general
data-driven framework for explainability of solutions in mathematical
optimization together with an analysis of the computational
complexity. We make it concrete for explainable
shortest paths. Computational results show the applicability of the novel framework. 

\subsection{Our Contributions}

We review the current state-of-the-art for explainability in AI and optimization in Section~\ref{sec:literature}.
We then show how to extend a given mathematical optimization model
by additional terms that explicitly model the explainability of a solution in
Section~\ref{sec:model}.
Adding explainability as a second criterion in the objective function,
the explainable model can interpolate between optimality and explainability. 
Utilizing today's availability of vast amount of data, we address
explainability by incorporating similarity-based explanations into the optimization framework.
We use easily comprehensible features of instances and solutions of
past observations to explain why a similar
solution should be employed. 
This way we are not explaining what aspects of the solution make it a ``good'' solution, but provide explanations for particular instance-solution pairs, exploiting the availability of historic data.
This moves the emphasis from provably optimal solutions to an
explanation in the decision-making process: having an explanation at hand might be more (or just as) valuable as being able to prove the solution's optimality. Note that we assume a set of comprehensible features as part of our framework's input, without exploring the specifics of what renders a feature comprehensible.
Although we prove in Section~\ref{sec:discussion} that straight-forward explainable optimization
models are often NP-hard, we show that for a polynomial optimization
problem its explainable version remains
polynomially solvable in case features are modeled by a set of 'important'
elements in a solution. This is in particular true
for the explainable shortest path problem with features represented by a
specific subset of edges. 
We present computational experiments using artificial and real-world data for the shortest path problem in Section~\ref{sec:exp}.
Our results demonstrate that enforcing an explainable path only
increases its cost by a small amount. The paper is concluded in Section~\ref{sec:conclusions}, where we also point out future research directions.

\section{Related Literature}
\label{sec:literature}

\paragraph{Explainability in AI.}
In the context of machine learning the question of human-understandable, explainable artificial intelligence (AI) is heavily discussed in the literature.
Especially for deep learning, methods for visualizing, explaining and interpreting the models and outputs has recently attracted increasing attention, see e.g.~the book by \citet{explainable_ai_book_deep_learning} or the survey by \citet{explainable_ai_survey}.

Recent trends include formal explanations \citep{marques2022delivering}, data-driven methods \citep{li2020survey} and counterfactual analysis \citep{counterfactual_explain}.
By understanding the contribution of specific features of the input, domain experts and practitioners can better comprehend how the AI system arrives at specific predictions and feature-based explainability techniques offer insights into relevant input features that support the system prediction \citep{explainable_deep_learning_lrp, explainable_deep_learning_lrp_code}. 
Hence there is a need to decide whether a specific feature contributes to a solution \citep{huang2023solving}. While interpretable methods are gaining popularity \cite{rudin2019stop}, the pursuit of providing easily understandable explanations continues to be an essential focus of research. Nevertheless, there is an inevitable price of explainability when forcing the solution to be easily comprehensible \citep{dasgupta2020explainable,gupta2023price}.

\paragraph{Explainability in Optimization.}

The NP complexity class allows for checking the feasibility and quality of solutions in polynomial time.
This property enables easy comparison between different solutions, known as contrastive explanation~\citep{miller2021contrastive}.
However, this type of explanation assumes that the inquirer has an alternative solution in mind.
In complex domains, it is often unlikely that such an alternative solution exists, which raises the question of why a particular solution was chosen.
Furthermore, we typically address this question by referring to the specific model, algorithm, or heuristic used, but this may not be satisfactory for non-experts.

In the field of operations research and management science, recent approaches have been published regarding explainability in scheduling and planning.
\citet{collins2019towards} propose approximate post-hoc argumentation-based explanations for planning problems by extracting and abstracting causal relationships.
\citet{oren2020argument} present a tool that utilizes formal argumentation and dialogue theory to explain plans to non-technical users. \citet{vcyras2019argumentation} address explainability for makespan scheduling problems by using an abstract argumentation to extract explanations related to feasibility and efficiency.
This approach is expanded by~\citet{vcyras2021schedule}, resulting in a tool that provides interactive explanations in makespan scheduling.
In a more general approach, \citet{erwig2021explainable} introduce a method for explaining solutions obtained through dynamic programming, focusing on the explainability of the program itself.
Additionally, in the realm of multi-objective optimization, there is a growing emphasis on providing comprehensible insights.
\citet{sukkerd2018toward} generate verbal explanations in a planning environment that clarify the trade-offs made to reconcile competing objectives. \citet{corrente2021explainable} employ simple decision rules to record decision makers' preferences and iteratively converge towards the best compromise solution on the Pareto front, providing insights into the impact of given answers.
Furthermore, an explainable interactive multi-objective optimization method that supports decision makers in expressing new preferences to improve desired objectives was introduced by~\citet{misitano2022towards}.

Lately, there has been a growing interest in data-driven approaches that aim to improve the explainability and interpretability of optimization problems.
\citet{forel2023explainable} provide a counterfactual explanation methodology tailored to explain solutions to data-driven optimization problems.
This way they can explore what would have happened if certain parameters were different from what actually occurred, making it possible to reason about alternative possibilities and hypothetical scenarios to explain the causal relationship between instances and solutions. Furthermore,  
\citet{GOERIGK2023Interpretable} offer a framework providing comprehensible optimization rules that leverage anticipated scenarios, thereby rendering the optimization process inherently comprehensible.

\section{An Optimization Model for Explainable Solutions \label{sec:model}}

\subsection{Explainability: Motivation and Discussion}

The availability of simple optimality certificates
is highly desirable, in particular when human beings need to understand
why some decisions shall be taken or when they have to remember how an
action shall be performed. This can be very relevant in emergency or
safety-critical settings or in situations when time is tight.
It may have legal repercussions when
decision makers need to be able to justify their choices.
Despite its relevance, however,
only for a very limited number of optimization models one can
directly understand why a solution is optimal. We mention a few such
settings next.

For example, let us consider the simple cardinality constrained
minimization problem. For selecting $p$ out of $n$ many items with
minimal cost, it suffices to know the $p$-th smallest cost and
partition the items based on this threshold. 
Furthermore, for optimization problems where some duality theory can be applied and
strong duality holds,
one could use the optimal solution of the dual problem as a
certificate model. For example, for a flow solution in a maximum flow
network, the corresponding minimum cut yields a certificate that the
flow is indeed optimal. Beyond this, there exists a variety of
optimality certificates like variational inequalities or optimality
systems. 
Unfortunately, such certificates are typically as 
complex to understand as the original, primal, solutions and thus
cannot be called an explanation.
This is particularly the case when a solution choice needs to be explained
to somebody without background knowledge in optimization.

As a result, it seems necessary to explicitly model the explainability
criterion as part of an optimization model, as standard models typically
do not yield explainable solutions. 
In the era where data abound, we model the explainability criterion
by extending a standard optimization model by a
data-driven component. Motivated by successful applications in explainable AI, we
rely on the usage of a database of historical instances together with
previous decisions.
We consider a solution explainable
if it differs not too much from accepted historical data.
We
thus use the available data to construct solutions that are similar to the solutions of
the most similar known instances.
The proposed framework thus interpolates between optimality and
explainability.

\subsection{The Framework}

In the following, we write vectors in bold, and use the notation $[n]:=\{1,\ldots,n\}$ to denote index sets.

We study a fixed underlying problem type for which $\I$ is the set of all instances and $\mathcal{X}\subseteq \mathbb{R}^n$ is the general domain. For each instance $I\in \I$, let $\XI \subseteq \mathbb{R}^n$ be the corresponding solution space. 
We call $f^I:\XI \rightarrow \mathbb{R}$ the objective function of $I$   and $f^I(\pmb{x})$ the objective value of $\pmb{x}\in\XI$.
Together they constitute the \emph{nominal} optimization problem
\begin{equation}
	\min_{\x \in \X(I)} f^I(\x). \label{Eq:nom_problem} \tag{Nom}
\end{equation}

A data-driven framework for explainability in general heavily relies on the existence of high quality data consisting of historic instance-solution pairs. Assume we have $N>0$ data points $(I^i,\pmb{x}^i,\lambda^i)$, where $I^i \in \I$ is a full description of the $i$-th historic instance, $\pmb{x}^i\in \mathcal{X}(I^i)$ is the employed solution, and $\lambda^i\in[-1,1]$ is a confidence score that solution $\pmb{x}^i$ is considered optimal in instance $I^i$.
We consider feature functions $\phi_\I:\I\rightarrow F_\I$ and $\phi_\X:\I \times \X \rightarrow F_\X$ 
that aggregate information of instances as well as solutions within features spaces $F_\I \subseteq \mathbb{R}^p$ and $F_\X \subseteq \mathbb{R}^q$, respectively. We define two metrics $d_{\I}:F_\I \times F_\I \rightarrow \mathbb{R}_+$ and $d_{\X}:F_\X \times F_\X \rightarrow \mathbb{R}_+$ as similarity measures for instances and solutions, respectively.

A key assumption is that the used
features indeed represent easily comprehensible aspects of both the
instance and the solution. Due to the subjectiveness of explainability
itself, we cannot quantify this property in general. However,
in the example
below and in the computational results
we provide some possibilities for choosing features. 
We consider a new instance $I \in \I$ for which we want to find an explainable solution. Let $$S_\epsilon(I)=\{i \in [N] \mid d_\I(\phi_\I(I),\phi_\I(I^i)) \leq \epsilon\}$$
be the set containing the most similar historic instances with threshold $\epsilon\geq \min_{i \in [N] } d_\I(\phi_\I(I),\phi_\I(I^i)) $.
For
$\beta \geq 0$, we propose the following bicriteria optimization model:
\begin{equation}
	\min_{\pmb{x} \in \XI} \ \left\lbrace f^I(\pmb{x}), \sum_{i \in S_\epsilon(I)}\frac{\lambda_i d_{\X}\left( \phi_{\X}(I,\x),\phi_{\X}(I^i,\x^i)\right)}{1+\beta d_\I\left(\phi_\I(I),\phi_\I(I^i)\right)}
	\right\rbrace \tag{Exp}\label{exp_formula}
\end{equation}

The first objective is the objective function of the nominal problem. The second objective represents the explainability of solution $\pmb{x}$. To calculate this value, we consider all similar historic instances $S_{\epsilon}(I)$. For each such instance $I^i$, if the confidence $\lambda_i$ is positive, we would like to achieve a solution that is similar to the historic solution $\pmb{x}^i$. Similarity is measured in the solution feature space $F_{\X}$ using the distance $d_{\X}$. If $\lambda_i$ is negative, then there is an incentive to choose a solution that is dissimilar to $\pmb{x}^i$. The factor $\beta$ is used to adjust the influence of less similar historic instances on the explainability.  Our definition of explainability means that a decision maker can point to historic examples, and explain the current choice based on similar situations in which similar solutions were chosen.

In principle, any method to treat bicriteria methods can be applied to calculate Pareto efficient solutions to problem~\eqref{exp_formula}. In the following, we focus on the weighted sum scalarization. In this setting, both objective functions are added together, where different weights are utilized to find different Pareto efficient solutions. If the nominal objective is multiplied with some scalarization weight $\alpha \in (0,1)$ and the
explainability objective with $1-\alpha$, we can define $\tilde{\lambda}_i = \lambda_i / (1+\beta d_\I(\phi_\I(I),\phi_\I(I^i)))$ and thus obtain the following weighted sum problem formulation:
\begin{equation}
	\min_{\pmb{x}\in\X(I)} \alpha f^I(\pmb{x}) + (1-\alpha) \sum_{i\in [N]} \tilde{\lambda}_i d_{\X}(\phi_{\X}(I,\pmb{x}),\phi_{\X}(I^i,\pmb{x}^i)) \tag{WS-Exp}\label{wsexp}
\end{equation}
where $\tilde{\lambda_i} = 0$ for all $i \in [N]\setminus S_\epsilon(I)$.

\begin{example}
	Consider a knapsack problem for packing a suitcase for a vacation. An instance $I$ contains all relevant
	data such as object profits $\pmb{p}\in \mathbb{R}^n_+$, weights
	$\pmb{w}\in \mathbb{R}^n_+$ and the knapsack capacity $C \in \mathbb{R}_+$, but also available metadata. The
	solution space is $\XI=\mathcal{X}(\pmb{p},\pmb{w},C)=\{\pmb{x}\in \{0,1\}^n \mid \sum_{j\in[n]} w_jx_j \leq C\}$, i.e., for each object $j$ it
	is decided whether $j$ is included in the knapsack ($x_j=1$)
	or not ($x_j=0$).
	The optimization objective  is
	$f^I(\x)=-\sum_{j\in[n]} p_j x_j$.
	Possible instance features could be
	the knapsack capacity $C$, profits and weights of particular (or all) items, as well as information on metadata such as weather forecast, season, trip duration or type of vacation (skiing, hiking, beach, etc.). Possible solution features might be the
	number of packed items of specific item groups (number of snacks, shoes, pants ...), the overall number of packed items, or the overall weight of items.
\end{example}

\section{Model Discussion\label{sec:discussion}}

The computational complexity of~\eqref{wsexp} depends on the
difficulty of the nominal optimization problem~\eqref{Eq:nom_problem},
the specific choices of the feature function~$\phi_{\X}$ and the
metric~$d_\X$ on the solution feature space~$\X$. 
In this section, we present some insights into the
computational complexity of the explainable optimization framework.

Unfortunately, even very simple cases of finding explainable solutions by solving the optimization problem~\eqref{wsexp} can be hard.

\begin{theorem}\label{th:hardness}
	Problem~\eqref{wsexp} is NP-hard and not approximable, even if $N=1$, $\phi_{\X}$ is an affine linear function that maps to $\mathbb{R}$, $d_{\X}$ is the absolute difference, $f^I(\pmb{x})=0$, and $\X(I)=\{0,1\}^n$.
\end{theorem}
\begin{proof}
	Consider the partition problem, which is known to be NP-complete \citep{garey1979computers}: Given integers $a_1,\ldots,a_n$, is it possible to determine a set $S\subseteq[n]$ such that $\sum_{j\in S} a_j = \sum_{j \notin S} a_j$?
	Given an instance of the partition problem, set $\phi_{\X}(I,\pmb{x}) = \sum_{j\in[n]} a_j x_j - A/2$, where $A=\sum_{j\in[n]} a_j$.
	We assume that there is one historic instance $I^1$ given with $\phi_{\X}(I^1,\pmb{x}^1) = 0$ and $\lambda_1=1$ (note that it is trivial to construct such an instance).
	Then, the problem
	\[ \min_{\pmb{x}\in\{0,1\}^n} (1-\alpha)\tilde{\lambda}_1 \left|\sum_{j\in[n]} a_jx_j - A/2\right| \]
	has a solution with objective value equal to zero if and only if the partition instance is a yes-instance.
\end{proof}

While Theorem~\ref{th:hardness} shows that \eqref{wsexp} is hard already in simple situations, there is a class of problems where it is possible to retain polynomial solvability, if the nominal problem can be solved in polynomial time. We assume that the set of feasible solutions $\X(I)$ contains binary vectors and does not depend on the instance, and write $\X$ instead. Let $c^I_j$ be the costs of item $j\in[n]$ in instance $I$.

We consider a specific type of feature mapping $\phi_{\X}$ that is
natural in this case: Let us assume that there is a set of important
elements $E'\subseteq [n]$. For example, this could be a set of bridges that span a
river in a shortest path problem that is well-known to be polynomial-time
solvable.
It is also possible to define $E'=[n]$ to include complete solutions in this comparison.
We define the feature space $F_{\X}$ to be $\{0,1\}^{|E'|}$, that is, $\phi_{\X}$ is a mapping where for every element in $E'$, we use an indicator binary value to represent if a solution $\pmb{x}$ contains this element or not. We refer to this mapping as $\chi_{E'}(\pmb{x})$. Let $H: \{0,1\}^{|E'|} \times \{0,1\}^{|E'|} \to \mathbb{N}$ denote the Hamming distance.

\begin{theorem}\label{th:poly}
	Problem
	\begin{equation}
		\min_{\pmb{x}\in\X} \alpha \sum_{j\in [n]} c^I_j x_j + (1-\alpha) \sum_{i\in[N]}\tilde{\lambda}_i H(\chi_{E'}(\pmb{x}),\chi_{E'}(\pmb{x}^i)) \tag{WS-Exp$^H$}\label{spexp}
	\end{equation}
	can be solved in polynomial time, if the nominal problem can be solved in polynomial time for arbitrary cost values.
\end{theorem}
\begin{proof}
	We note that
	\begin{align*}
		&\alpha \sum_{j\in [n]} c^I_j x_j + (1-\alpha) \sum_{i\in[N]}\tilde{\lambda}_i H(\chi_{E'}(\pmb{x}),\chi_{E'}(\pmb{x}^i)) \\
		&= \alpha\sum_{j\in [n]} c^I_j x_j + (1-\alpha)\sum_{i\in[N]} \tilde{\lambda}_i \sum_{j\in E'} |x_j - x^i_j| \\
		&= \alpha\sum_{j\in [n]} c^I_j x_j \\
		& \quad + (1-\alpha) \sum_{i\in[N]} \tilde{\lambda}_i \left( \sum_{j\in E': x^i_j = 0} x_j +
		\sum_{j\in E': x^i_j = 1} (1-x_j) \right)
	\end{align*}
	For all $j\in [n]$, let
	\[ \mu_j := \begin{cases} \sum_{i\in[N] : x^i_j=0} \tilde{\lambda}_i - \sum_{i\in[N]: x^i_j = 1} \tilde{\lambda}_i & \text{ if } j\in E' \\
		0 & \text{ else }\end{cases}\]
	With this notation, the objective function thus becomes $\sum_{j\in [n]} (\alpha c^I_j+(1-\alpha)\mu_j)x_j + \sum_{i\in[N]} (1-\alpha)\tilde{\lambda}_i \cdot |\{j\in E' : x^i_j = 1\}|$. As the nominal problem with arbitrary costs can be solved in polynomial time, the claim follows.
\end{proof}

Let us consider in more detail the shortest path problem. There,
a directed graph $G = (V, E)$ with vertices $V$ and
edges $E\subseteq V\times V$ with non-negative edge weights
$c_{u,v} \geq 0$ for $(u, v) \in E$ are given. 
The task consists in determining a cost-minimimal path between two
vertices $s,t\in V$.

\begin{cor}
	Problem~\eqref{spexp}
	can be solved in polynomial time for the shortest path problem in directed acyclic graphs and for the minimum spanning tree problem.
\end{cor}

In particular the shortest path problem will be studied in more detail in the computational section. We note that the problem still remains hard for general directed graphs $G=(V,E)$.

\begin{theorem}\label{th:sphard}
	Let $\X$ be the set of all $s$-$t$ paths. Problem~\eqref{spexp} is NP-hard and not approximable, even if $\tilde{\lambda}_i \ge 0$, $c^I_e=0$, $|E'|=1$ and $N=1$.
\end{theorem}
\begin{proof}
	Consider the 2-vertex-disjoint path problem, which is NP-complete, see \cite{fortune1980directed}. Given a directed graph $G=(V,E)$ with nodes $s_1,s_2$ and $t_1,t_2$, is there an $s_1$-$t_1$ path and an $s_2$-$t_2$ path which do not share any vertex? Given an instance of this problem, we construct an instance of \eqref{spexp}. We set $\bar{G}=(V,\bar{E})$ with $\bar{E} = E \cup \{\bar{e}\}$ with $\bar{e}=\{(t_1,s_2)\}$. Set $x^1_e=0$ for all $e\in E$ and $x^1_{\bar{e}} = 1$. Further set $c^I_e = 0$ for all $e\in\bar{E}$ and $\tilde{\lambda}_1 = 1$. Then there is an $s_1$-$t_2$ path in $\bar{G}$ with costs 0 if and only if the 2-vertex-disjoint path instance is a yes-instance (recall that a path may not repeat a vertex or an edge).
\end{proof}

The hardness of Theorem~\ref{th:sphard} can be avoided if we allow paths to re-visit nodes and edges. Recall that a sequence of consecutive nodes and edges is called an edge progression. As before, we define the Hamming distance as the sum of binary indicators if an edge is present in one edge progression but not the other.

\begin{theorem}
	Let $\X$ be the set of $s$-$t$ edge progressions. Problem~\eqref{spexp} can be solved in polynomial time if $|E'|$ is constant and $G$ does not contain negative cycles with respect to $\pmb{c}^I$.
\end{theorem}
\begin{proof}
	Enumerate all subsets of $E'$. For each subset, we remove the edges $E\setminus E'$ from the graph and enumerate all possible permutations of elements. As $|E'|$ is constant, there is a constant number of such permutations. For each permutation $((s_1,t_1),\ldots,(s_k,t_k))$, calculate a shortest path with respect to $\pmb{c}^I$ from each $t_i$ to $s_{i+1}$, as well as from $s$ to $s_1$ and from $t_k$ to $t$. Because there are no negative cycles, this can be done in polynomial time. Concatenate these paths with the edges of the permutation to find an $s$-$t$ edge progression. By taking the minimum over all such edge progressions for all permutations of a subset, we find the minimum cost $s$-$t$ edge progression with respect to $\pmb{c}^I$ that includes the elements of the edge subset. By repeating the process for each possible subset of $E'$, we find a minimizer of the cost function of \eqref{spexp}.
\end{proof}

Let us end with some remarks on the relation between~\eqref{Eq:nom_problem} and its explainable version~\eqref{wsexp}.
As argued before, the
nominal~\eqref{Eq:nom_problem} has to be extended by additional
terms in the objective function for addressing the solution explainability.
Theorem~\ref{th:sphard} states that, already for easy settings, \eqref{wsexp} is NP-hard, even if a polynomial-time solution algorithm is known for~\eqref{Eq:nom_problem}.
Nevertheless, for specific optimization problems, it may very well be true that its explainable version~\eqref{wsexp} is an
algorithmically tractable problem as well.
This is true, for example, for continuous but convex optimization problems~\eqref{Eq:nom_problem}, where $\tilde{\lambda}_i\ge 0$ and $\phi_{\X}, d_{\X}$ are given by convex functions.

\section{Computational Experiments}\label{sec:exp}

We illustrate our framework by making it concrete for the 
shortest path problem.
For a directed graph $G = (V, E)$ with edge weight $c_{u,v}\geq 0$ for
edge $(u,v)\in E$, the task consists of determining a shortest path
from some node $s\in V$ to a node $t\in V$ with respect to the weights that model e.g. travel times between two connected nodes.
In addition, each edge $(u, v) \in E$ also has a
(fixed) non-negative length
$l_{u, v} \geq 0$ that models e.g. the geographical distance of two nodes in a street network.
As instance features for a given shortest path instance $I\in \I$ we consider
the edge weights $(c_{u,v})_{(u,v) \in E_{\I}}$  of a particular set of edges $E_{\I}\subseteq E$ together with the $s$-$t$ pair of the instance.
The similarity score $d_\I^i\coloneqq d_\I(\phi_\I(I),\phi_\I(I^i))$ between the current instance $I$ and an instance $I^i$ from the data set is set to infinity, if they have different $s$-$t$  pairs, and to the Euclidean distance of the edge weight vectors in case of equal $s$-$t$ pairs (for this reason, we conduct all experiments using a fixed $s$-$t$ pair).
For $(u,v)\in E$, let variable $x_{u,v}\in \{0,1\}$ indicate whether $(u, v)\in E$ is part of the path or not.
As solution features we choose $\phi_{\X}(I,\x)=(l_{u,v}\cdot x_{u,v})_{(u,v)\in E_{\X}}$ for a subset $E_{\X}\subseteq E$ and measure $d_{\X}$, the similarity of solutions, by calculating the Manhattan distance regarding these edges (using this definition instead of the (unweighted) Hamming distance ensures that results are less dependant on the graph representation).

Following Theorem
\ref{th:hardness}, the explainable shortest path
problem in graphs that contain cycles is NP-hard. We thus next
introduce a mixed-integer linear programming (MIP) formulation of \eqref{wsexp}.
Let $\pmb{x}^i \in\{0,1\}^E$ be a binary vector, indicating the shortest path of
similar instance $i \in S_\epsilon(I)$. 
In our experiments, we set the confidence score $\lambda_i$ and the scale-factor $\beta$
to a constant value of one.
Hence, $\tilde{\lambda}_i\coloneqq(1+d_\I^i)^{-1}$ for $i \in S_\epsilon(I)$ to weigh the influence according to the similarity of the instances. 
Furthermore, we always consider the five most similar historic instances, i.e., we use $|S_\epsilon(I)|=5$ (other settings were tested in preliminary experiments, which resulted in qualitatively similar results).
The model for explainable shortest paths~\eqref{wsexp} then reads
\begin{align*}
	\min \quad &\alpha\sum_{(u, v) \in E} c_{u, v} \cdot x_{u,v}\\ &+(1-\alpha)\sum_{i\in S_\epsilon(I)}\tilde{\lambda}_i\sum_{(u, v) \in E_{\X}} {l_{u,v}}|x_{u,v}^i - x_{u,v}| \\
	\text{ s.t. }\quad&\sum_{v \in V} x_{s,v} = 1, \quad	\sum_{v \in V} x_{v,t} = 1, \\
	&	\sum_{v \in V} x_{u,v} - \sum_{v \in V} x_{v,u} = 0 \qquad \forall\  u \in V \setminus \{s, t\}, \\
	&	t_v\ge t_u+1+(|V|-1)(x_{u,v}-1)\quad \forall\  u,v \in
	V\setminus \{s,t\},\\
	&	t_i\ge0\qquad \forall i\in V,\\
	&	x_{u,v} \in \{0, 1\} \quad \forall (u, v) \in E.  \tag{Exp-MIP}\label{Exp-Path}
\end{align*}
The first three sets of equations model a path in
the network as a flow of value one from $s$ to $t$. Using the
classical
model by \citet{miller1960integer}, the subsequent set of variables and
inequalities enforce the solution path to be necessarily cycle-free.
By sequentially numbering  each node
of the path it is ensured that each node is visited at most once.
We note that the absolute value in the objective function can be easily resolved by a case distinction for the values of the parameters $x_{u,v}^i\in \{0,1\}$.

We call the first sum of the objective function of model \eqref{Exp-Path} the \textit{optimality
	value} and the second sum \textit{explainability value}.
We note
that in this definition, a small value means high explainability.

We next present numerical
experiments on an illustrative synthetic example as well as on a
real-world network with historical data.
All computations were carried out using a \textsc{Python} implementation on a machine with an AMD Ryzen 5 5600H 3.30~GHz processor and 16 GB RAM.
We utilized \textsc{Gurobi 10.0.1} as MIP solver~\cite{gurobi}.

\subsection{Explainable Shortest Paths for a Synthetic Testcase}

\begin{figure}[htb]
	\centering
	\includegraphics[width=.7\columnwidth]{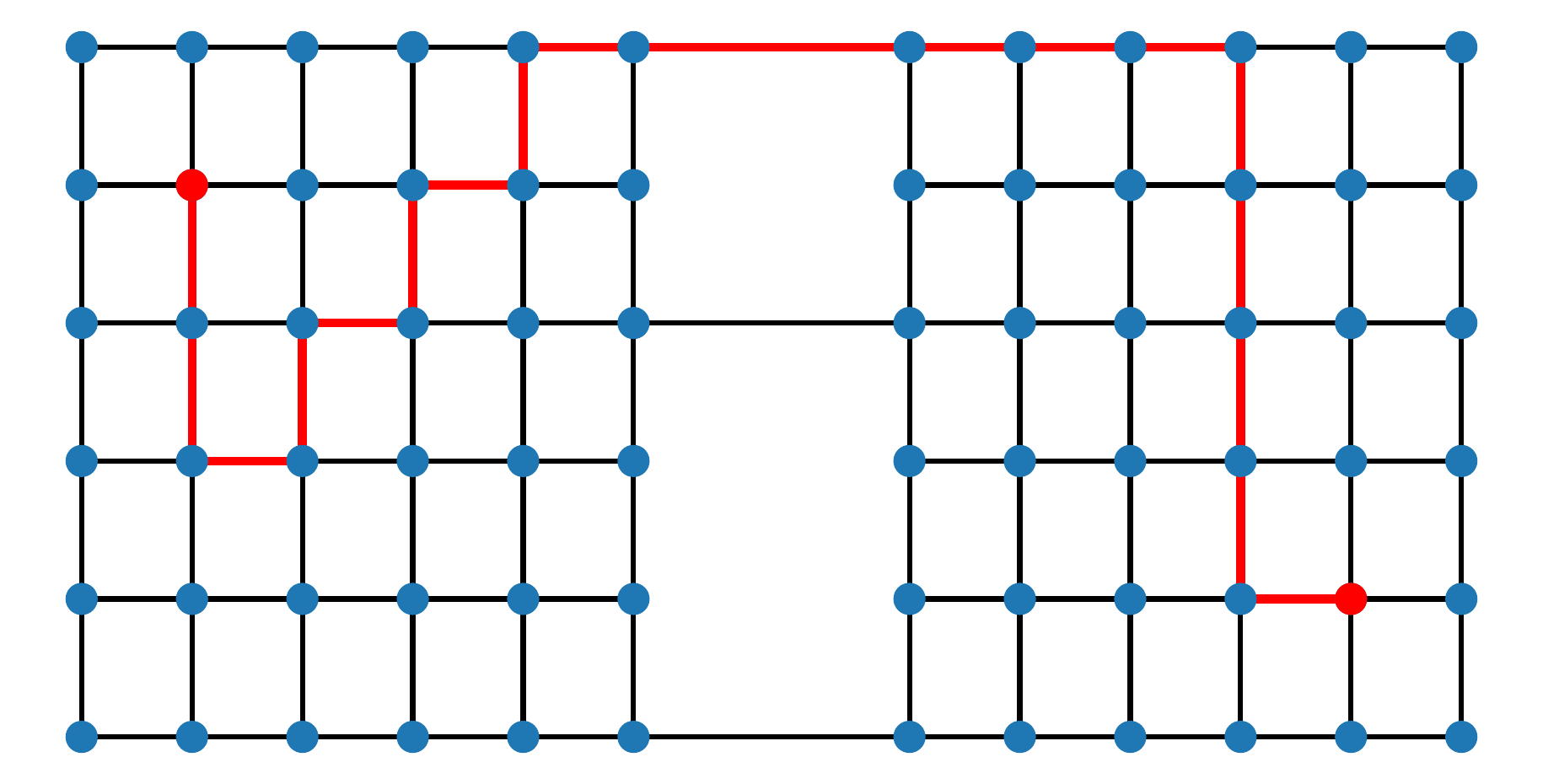}
	\caption{Regular grid graph with three bridges. A shortest
		path (red) connects the two red nodes.}
	\label{fig:graph_bridge}
\end{figure}

In order to illustrate explainable optimization, we study a setting
comprising two cities or two districts within some city that are interconnected by routes or bridges.
The graph used in our experiments is displayed in
Figure~\ref{fig:graph_bridge}. It consists of two $6\times6$ - grids
connected by three edges that represent bridges between them. The
shortest path connecting the two red nodes is drawn in red. 
The main motivation for this synthetic testcase is that the three connecting edges are well suited to define instance and solution features because one has to pass at least one of them to get from the left side to the right side of the graph.

We define random nominal edge weights with a uniformly distribution between 0 and 2.
These values thereby representing the mean travel time of the edges.
We also use these values for the distance parameters $l_{u,v}$ to construct solution features.
For the historical set of instances, we created $N=50$ instances by
varying the edge weights given the topology from Figure~\ref{fig:graph_bridge}.
To imitate the influence of different traffic situations and to
account for real-world variability, we perturbed the nominal edge
weights using a normal distribution with a standard deviation of
$\sigma$. In preliminary computational experiments, it turned out that
for very
small values $\sigma<0.5$, the resulting instances are very similar, whereas for
large values $\sigma > 5$ the weights are almost uncorrelated. We thus
use an intermediate value of
$\sigma = 2$ which yields interesting results.

We restrict the instance feature function $\phi_\I(I)= (c_{u,v})_{(u,v) \in E_{\I}}$ on the subset $E_{\I}\subseteq E$ that consists of the three bridge edges in the middle of Figure~\ref{fig:graph_bridge}. As solution features we consider all edges, i.e.~ $E_{\X}=E$ and hence the similarity of solutions depends on all used edges.

We calculate the Pareto front using a range
of $\alpha$ values spanning from 0 to 1.
The front is given by the corresponding optimality as well as the explainability
value for
each chosen value of $\alpha$.
We repeat the entire process 50 times and compute a
relative optimality score ($\frac{\text{optimality value}}{\text{best
		optimality value}}$) and relative explainability score
($\frac{\text{lowest explainability value}}{\text{explainability}}$).
Plotting these two scores for
each value of $\alpha$ against each other, we generate an average Pareto curve as in Figure~\ref{fig:pareto_synthetic}.
This average curve reflects the overall numerical results across different
traffic scenarios.

Per definition, an optimality value of 1.0 equals the optimum of the nominal
problem. Although the
nominal solution does not take explainability into account, on average
(bold curve) it already obtains a relative explainability value of around 0.7. By
allowing only a small increase in the value of the (nominal)
objective,
explainability improves strongly. Indeed, a high 
relative explainablity value of more than 0.9 is obtained by increasing the
objective value by only around 20\%. 

Additionally we examine the case where  $E_{\X}=E_{\I}$, i.e.~the similarity of solutions is only measured across the three bridge edges.
This only leads to an average relative reduction in optimality of $10\%$ in order to obtain fully explainable solutions.
The corresponding Pareto fronts predominantly consist (with very few exceptions) of only a single point, as the nominal optimal solution frequently exhibited the same features as the most similar data points. Here we should note that features need to be selected with great care: one can construct solution features such as ``an edge leaving node $s$ was used'' that can explain any solution perfectly,  yet lack any significance.

\begin{figure}[htb]
	
	\centering	
	\includegraphics[width=0.5\columnwidth]{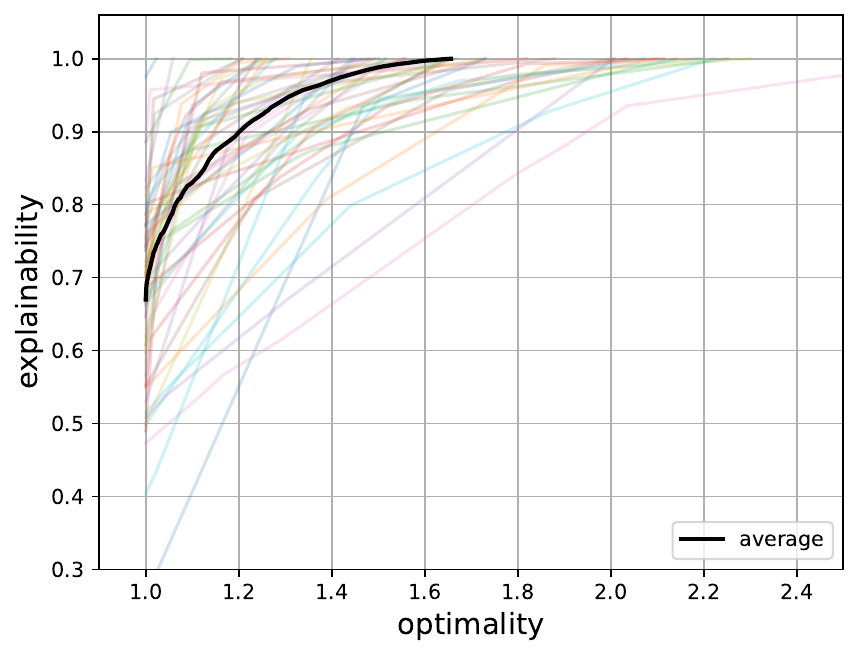}
	\caption{Pareto plot illustrating the relative scale of
		averaged results from 50 runs.
		The bold curve represents average values over
		all instances.}
	\label{fig:pareto_synthetic}
	
\end{figure}

\subsection{Explainable Shortest Path for a Real Street Network with
	Historical Data}

\begin{figure}[htb]
	
	\centering	
	\includegraphics[width=0.4\columnwidth]{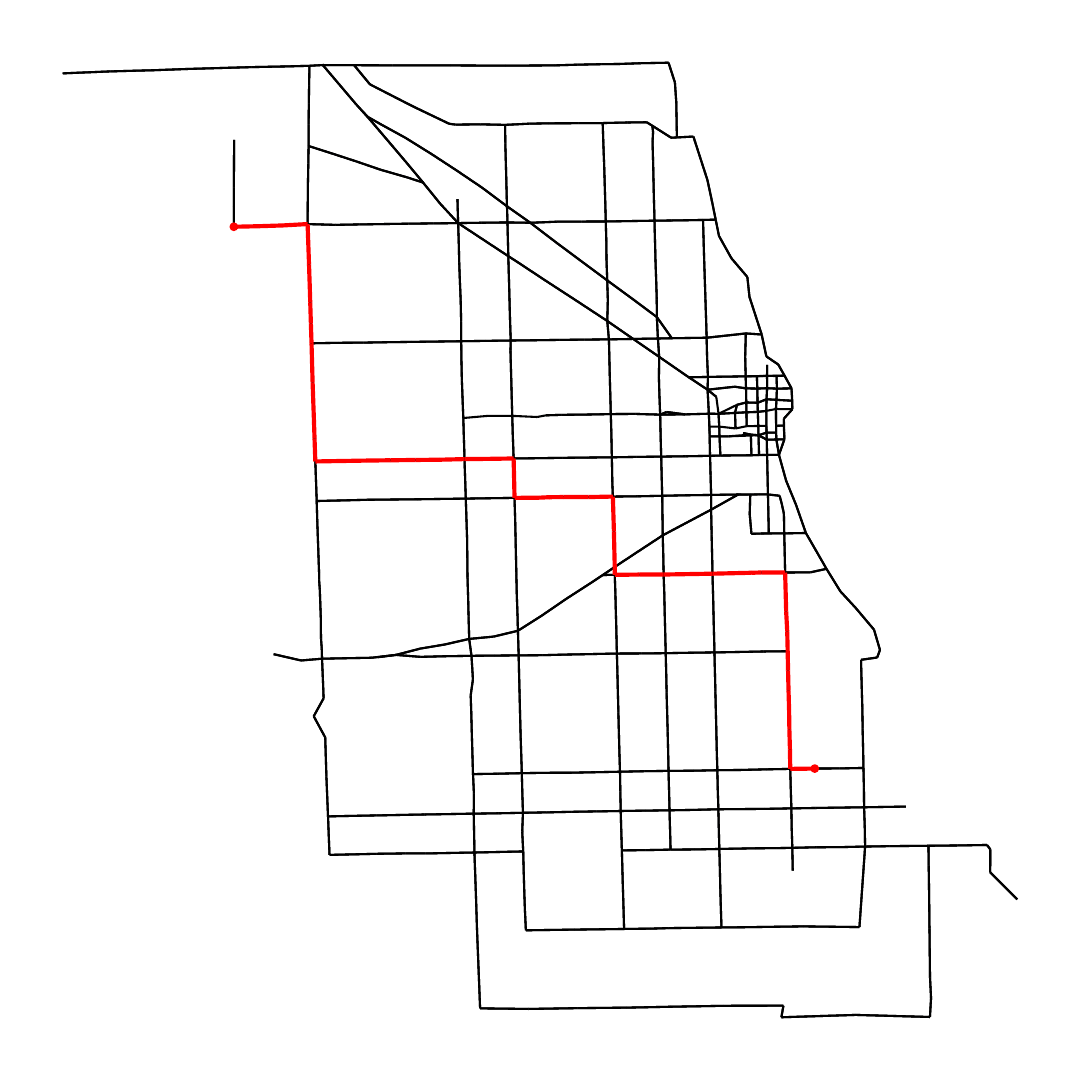}
	\caption{Graph of a real road network in Chicago. The 
		path between two red nodes is drawn in red.}
	\label{fig:graph_Chicago}
	
\end{figure}

For this case study, we utilize real-world data collected by bus
drivers from the city of Chicago, see
\cite{chassein2019algorithms}. By this real-world case study, we aim to gain deeper insights into
the evaluation of explainability versus optimality under realistic traffic scenarios.
The network topology consists of
538 nodes and 1287 edges, where each edge is associated with the coordinates of its start and end points. These nodes and edges represent the road network in the city.
We use 4363 different historical scenarios for the (average) edge
velocities that we set as edge weights.
The data was collected from March $28^{th}$ until May $12^{th}$, 2017,
where velocity data is recorded in time steps in 15-minute
intervals. This detailed data provides valuable insights into the
traffic patterns. We note that the data is not
complete for the entire duration (in total we are missing 67 time steps). This does not impact our calculations, as we isolate individual data points and never consider the continuous timeline itself in our analysis.
We solve the shortest path problem for a randomly selected instance. 
By utilizing the date and time information as additional instance features, we filter the
historical data taking solely velocities measured 
within a time window of up to 30 minutes before or after the current
daytime. 
This narrows down the dataset to about 170
relevant historical instances that closely align with the present
conditions.
For defining the instance and solution features, we use $E_{\I}=E_{\X}=E$.
Then we proceed by repeatedly solving \eqref{Exp-Path} for the individual values for $\alpha$ as we did 
in the synthetic testcase.
To obtain an average Pareto curve, we optimize again 50 randomly selected instances drawn from the dataset.
Solving the nominal shortest path problem can be performed very
quickly ($<$0.1s). But also for one choice of $\alpha$, the solution of \eqref{Exp-Path}
takes at most a couple of seconds even for our real-world instances, which is an affordable time period in many
practical settings.

\begin{figure}[htb]
	
	\centering	
	\includegraphics[width=0.5\columnwidth]{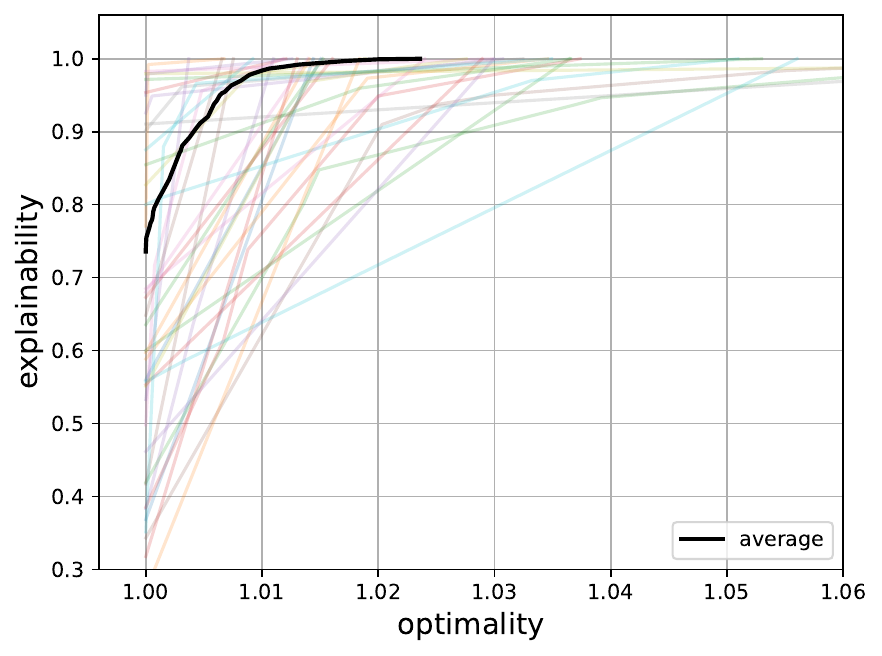}
	\caption{Analysis of 50 random samples explained by 4363 pre-filtered instances, aggregated by daytime, compared with 5 historical solutions.}
	\label{fig:pareto_Chicago}
	
\end{figure}

Similarly as in the synthethic grid network, 
the nominal solution on average
(bold curve) already obtains an explainability value of around
0.75. Again, only slightly increasing the value of the (nominal)
objective,
the corresponding value improves even more drastically than in the
synthetic case. Indeed, a fully explainable solution is
obtained by allowing an increase in objective value by less than
3\%. In summary, the synthetic together with the real-world case study
show that enriching the shortest path problem by explainability is 
practically doable. In addition, in our computations, the objective
values increase by only a small amount, showing that only a tiny
prize has to be paid for obtaining an explainable solution.

\section{Conclusions and Outlook}
\label{sec:conclusions}
We presented and theoretically substantiated a general, data-driven
framework that incorporates explainability into mathematical
optimization. A solution is considered explainable
when it shares similar features with solutions from similar favorable
instances and deviates in features from similar unfavorable
instances. Using artificial and real-world case studies on the explainable shortest path problem, we have demonstrated that solutions can be computed quickly with only a very small price of explainability.
Presumably, plots for
explainability versus optimality will exhibit similar characteristics
also for other optimization problems.

We only used historic favorable solutions as explanations. Although this leads to data-driven explainability, such solutions may not be considered fair when applied over and over again. This can happen in shift planning when always the same personnel is assigned to unpopular shifts. Obtaining fair solutions over time, for example in ressource allocation, has been studied in  \cite{lodifairness}. Furthermore, in machine learning efforts to bridge the gap between fairness and explainability were made \cite{zhao2023fairness}. 
There's also a trend on discovering diverse solutions \cite{arrighi2023synchronization}, which offers different choices at various stages, which can be used to uphold fairness. Exploring fairness issues using the explainable framework presented here will be explored in future research and can be addressed by appropriately adapting parameters $\lambda_i$ in model \eqref{exp_formula} such that similarities in unpopular or recent situations is avoided.

\section*{Acknowledgements}
The authors thank the DFG for their support within Projects B06 and
B10 in CRC TRR 154, as well as within Project-ID
416229255 - SFB 1411. 
Furhtermore, this work was supported by the Federal Ministry for Economic Affairs
and Energy, Germany, Grant 03EI1036A.
The authors acknowledge support from Schloss Dagstuhl, Seminar 22441.

\end{document}